\documentclass[12pt,twoside,a4paper]{article}

\usepackage[english]{babel}
\usepackage[latin1]{inputenc}
\usepackage[T1]{fontenc} 
\usepackage{dsfont}
\usepackage{amssymb,amsmath,amsthm,amscd,mathrsfs,bbm}
\usepackage[left = 1 in, right = 1 in, bottom = 1.5in]{geometry}

\usepackage{fancyhdr}
\usepackage{titling}
\usepackage[titletoc]{appendix}
\usepackage{hyperref}
\usepackage{verbatim}
\usepackage{caption}
\usepackage{subcaption}
\usepackage{xcolor}

\theoremstyle{plain} \numberwithin{equation}{section}
\newtheorem{thm}{Theorem}[section]
\newtheorem{lemma}[thm]{Lemma}

\newtheorem{cor}[thm]{Corollary}

\theoremstyle{definition}

\theoremstyle{remark}
\newtheorem{rem}[thm]{Remark}
\theoremstyle{remark}
\newtheorem{ex}[thm]{Example}

\providecommand{\abs}[1]{\left\vert#1\right\vert}
\providecommand{\norm}[1]{\left\Vert#1\right\Vert}

\DeclareMathOperator{\supp}{supp}

\DeclareMathOperator{\Var}{Var}
\DeclareMathOperator{\Cov}{Cov}

\newcommand{\Nat}{{\mathbb N}}

\newcommand{\R}{{\mathbb R}}
\newcommand{\Z}{{\mathbb Z}}
\newcommand{\K}{{\mathbb K}}
\newcommand{\Exp}{{\mathbb E}}
\newcommand{\Pro}{{\mathbb P}}
\newcommand{\T}{\mathcal{T}}
\newcommand{\e}{\mathcal{E}}

\newcommand{\Xno}{X_{n,\mathbf{0}}}
\newcommand{\Xny}{X_{n,y}}
\newcommand{\XnI}{X_{n,y_1}}
\newcommand{\XnII}{X_{n,y_2}}

\newcommand{\xio}{\xi_{\mathbf{0}}}
\newcommand{\xiy}{\xi_{y}}
\newcommand{\xiyI}{\xi_{y_1}}
\newcommand{\xiyII}{\xi_{y_2}}
\newcommand{\xixI}{\xi_{x_1}}
\newcommand{\xixII}{\xi_{x_2}}
\newcommand{\fn}{\hat{f}_{n}}
\newcommand{\fno}{\hat{f}_{B'_n}}
\newcommand{\fny}{\hat{f}_{B'_n+y}}

\newcommand{\fnyI}{\hat{f}_{B'_n+y_1}}
\newcommand{\fnyII}{\hat{f}_{B'_n+y_2}}
\newcommand{\fnxI}{\hat{f}_{B'_n+x_1}}
\newcommand{\fnxII}{\hat{f}_{B'_n+x_2}}

\newcommand{\fnYi}{\hat{f}_{B'_n+Y_i}}
\newcommand{\fnYj}{\hat{f}_{B'_n+Y_j}}

\newcommand{\fBn}{\hat{f}_{B'_n}}
\newcommand{\fBno}{\hat{f}_{B'_n}(\xi_{\mathbf{0}})}

\usepackage{xcolor}
\definecolor{ao}{rgb}{0.0, 0.35, 0.0}
\newcommand{\edit}[1]{\textcolor{black}{#1}}
\newcommand{\new}[1]{\textcolor{black}{#1}}
\title{Estimation of entropy for Poisson\\ marked point processes}
\author{Patricia Alonso-Ruiz and  Evgeny Spodarev}
\date{\small\today}
\begin{document}
\maketitle
\begin{abstract}
In this paper, a kernel estimator of the differential entropy of the mark distribution of a homogeneous Poisson marked point process is proposed. The marks have an absolutely continuous distribution on a compact Riemannian manifold without boundary. $L^2$ and almost surely consistency of this estimator as well as its asymptotic normality are investigated.
\end{abstract}


\section{Introduction}


The concept of entropy was introduced by Shannon 
\edit{in the context of} information theory~\cite{Sha48} and its origin lies in the classical Boltzmann entropy of thermodynamics. In Shannon's original paper, entropy was defined both for discrete and continuous distributions in $\R^d$. In the last case it is called \textit{differential entropy} and this notion can be naturally generalized 
%
%
as follows: Let $P$ be a probability distribution of a random element $X$ on an abstract measurable phase space \edit{$(M,\mu)$} with probability density $f$ with respect to 
\edit{$\mu$}. The entropy of $X$ is given by 
$$
{\cal E}_f=-\Exp_P \left( \log  f(X) \right)=-\int\limits_M f(x) \log f(x) \, \mu(dx),
$$
where the expectation $\Exp_P$ is taken with respect to the probability measure $P$. 

%
%
\edit{In this paper, we consider a homogeneous Poisson marked point process (MPP) with marks from a compact Riemannian manifold of dimension $p\geq 1$ without boundary that are assumed to be independent of the process, and investigate the differential entropy of the mark distribution $\e_f$. Our motivation for the study of this quantity is its applicability to detect inhomogeneities in materials modeled by MPPs such as fibre-reinforced plastics, where the direction of each fibre corresponds to a marks of the MPP. During the production process of such materials, the direction of the fibres may deviate from the predefined one and thus give rise to undesirable clusters or deformations. If the deviation is strong, a significant change on the (local) entropy of the directional distribution can be expected. Considering marks with values in a Riemannian manifold makes this method applicable not only to directions but to any other characteristic of interest, for instance fibre length or fibre curvature. Asymptotic properties of such an estimator are important in particular for hypothesis testing.}

In the present work, we propose a nonparametric plug-in \edit{estimator of} the differential entropy \edit{$\e_f$} based on~\cite{AL76}. \edit{It} requires estimating the density of the distribution of interest in a nonparametric way, \edit{which we perform by means of} \textit{kernel density estimation}. \edit{This} technique \edit{was} introduced for stationary sequences of real random variables by Rosenblatt~\cite{Ros56} and Parzen~\cite{Par62}, \edit{and} extended to stationary real random fields in~\cite{EM14}. In the case of finite samples of i.i.d. random vectors on the sphere, nonparametric kernel estimation methods have been studied in~\cite{CHW87,BRZ88} and extended to Riemannian manifolds in~\cite{Pel05,HR09}. Alternative nonparametric estimators for the directional distribution in line and fibre processes have been presented in~\cite{Kid01}.

The main result of our paper, Theorem~\ref{thm CLT EE}, gives a central limit theorem (CLT) for an estimator of the differential entropy of the mark distribution density $f$ of a homogeneous Poisson MPP as the observation window grows to $\R_+^d$ in a regular manner. This result is an application of a more general result (c.f. Corollary~\ref{cor CLT mdep random sum}) of this type for sequences of $m_n$-dependent random fields proved in Section~\ref{sectEntrCLT}.

The paper is organized as follows: notation and basics of the theory of MPPs are given in Section \ref{sectPrelim}. In Section \ref{sectKernelEstim} we construct a nonparametric kernel density estimator of $f$ and give conditions for its $L^2$ and almost sure consistency. In Section \ref{sectEntrEst} we introduce the nonparametric estimator $\widehat{\e}_f(B_n)$ of the entropy ${\e}_f$
in an observation window $B_n\subset\R^d$ and prove \edit{its} $L^2$-consistency 
when the window size grows appropriately. Finally, we present in Section \ref{sectEntrCLT} a CLT for random sums of $m_n$-dependent random fields (cf. Corollary \ref{cor CLT mdep random sum}) where independence between the random number of summands and the summands themselves is not assumed. A special case of this result is applied to obtain a CLT of the entropy estimator.

\section{Preliminaries} \label{sectPrelim}
In this section, we briefly review basic notions from the theory of marked point processes. For an introduction and summary on these and other models of stochastic geometry we refer the reader to e.g.~\cite{SKM87,Spo13}.

\subsection{Poisson marked point processes}\label{sectMPPP}
In the following, 
$\Pi:=\{Y_i\}_{i\geq 1}$ 
\edit{will denote a} \textit{homogeneous Poisson point process} on $\R^d$ of intensity $\lambda>0$
\edit{and} $(M,g)$ a compact smooth Riemannian manifold of dimension $p$ without boundary and with Riemannian metric $g$. We further assume that $(M,g)$ is complete, i.e. $(M,d_g)$ is a complete metric space, where $d_g$ denotes the geodesic distance induced by the Riemannian metric $g$. The associated Riemannian measure will be denoted by $\upsilon_g$. A detailed construction of this measure can be found e.g. in~\cite[p. 61]{Sak96}. Note that since $M$ is compact, $\upsilon_g(M)$ is finite.

\medskip

To each point $Y_i\in\Pi$ we attach a mark $\xi_i\in M$ and assume that 
marks are i.i.d. random variables independent of the location of the points in $\Pi$. 
\edit{The \textit{Poisson marked point process $\Psi:=\{ (Y_i,\xi_i),~Y_i\in\Pi\}$} we will work with is} 
a random variable with values in $\mathcal{N}:=\{\varphi\text{ locally finite counting measure on }\R^d\times M\}$. An important property of this process is \textit{stationarity}, meaning that $T_y\Psi\stackrel{d}{=}\Psi$ for all $y\in\R^d$, where the translation operator $T_y$ is defined as 
$T_y \varphi(B\times L):= \varphi(\edit{(B+ y)}\times L)$ for any Borel set $B\times L\subset\R^d\times M$ and $\varphi\in\mathcal{N}$. We will assume that the 
\edit{distribution} of a typical mark $\xio$ has a density $f\colon M\to\R$ with respect to the Riemannian volume measure $\upsilon_g$. 

\edit{
\begin{ex}
Poisson fibre process (c.f.~\cite[Section 8]{SKM87}). A \textit{fibre} $F\colon [0,1]\to\R^2$ is a sufficiently smooth simple curve of finite length and a \textit{fibre process} $\Phi$ is a random closed subset of $\R^2$ that can be represented as the union of at most countable many fibres $F$. To each fibre, we can attach a mark $\xi_F\in[0,\ell]$ that represents its (random) length. If the fibre process is Poisson distributed, then $\Psi=\{(F,\xi_F), F\in\Phi\}$ and $M=[0,\ell]$.
\end{ex}
}

\edit{
\begin{ex}\label{example boolean}
Boolean model. Assume $d\geq 3$ and consider for each $1\leq k\leq d-1$ the Grassmannian $G(k,d)$, i.e. the set of all non-oriented $k$-dimensional flats in $\R^d$ that contain the origin  (see e.g.~\cite[p. 186]{Sak96}). This is a compact manifold of dimension $k(d-k)$. Furthermore, denote by $B(o,r)$ the ball of radius $r$ centered at the origin $o\in\R^d$. The homogeneous Poisson point process $\Pi\subset\R^d$ leads to the Boolean model
\[
\Phi:=\bigcup_{Y_i\in\Pi}\big((B(o,R_i)\cap Z_i)+Y_i\big),
\]
where $R_i$ and $Z_i$ are independent copies of the random radius $R\colon\Omega\to [0,r]$ and the random Grassmannian $Z\colon\Omega\to G(k,d)$, respectively. The particular case $k=d-1$ is used in applications to model lamellae structures, whereas the case $k=1$ corresponds to a Poisson fibre process with straight fibres. In both cases, $G(k,d)$ is isomorphic to the half-sphere $S_+^{d-1}$. Based on this model, one can directly work with the MPP $\Psi=\{(Y_i,B(o,R_i)\cap Z_i)\}_{i\geq 1}$, with $M=[0,r]\times G(k,d)$ and $p=k(d-k)$. Here, one may be interested in the entropy of some specific characteristics of the grains, for instance their radius $R$ and direction $Z$. 
\end{ex}
}
\subsection{Space of marks}\label{subsec space of marks}
Since our mark space is a manifold, we 
\edit{recall in this section} some useful concepts from Riemannian geometry. For further details we refer to~\cite{Boo86,Sak96}. 

\medskip

Let $\T_\eta M$ denote the tangent space of $M$ at $\eta\in M$ and let $\exp_\eta\colon \T_\eta M\to M$ denote the exponential map. For any $r>0$, $B_M(\eta,r):=\{\nu\in~M~\vert~d_g(\nu,\eta)<r\}$ defines a neighborhood of $\eta$, that we call a \textit{normal neighborhood of $\eta$} if there exists an open ball $V\subset \T_\eta M$ such that $\exp_\eta\colon V\to B_M(\eta,r)$ is a diffeomorphism.
The \textit{injectivity radius of} $M$ is defined as $\operatorname{inj}_g M:=\inf_{\eta\in M}\sup\{r\geq 0~\vert~B_M(\eta,r)\text{ is a normal nbhd. of }\eta\}$. 

\medskip

Let $U$ be a normal neighborhood of $\eta\in M$ and let $(U,\psi)$ be the induced exponential chart of $(M,g)$. For any $\nu\in U$, the \textit{volume density function} introduced by Besse in~\cite[p.154]{Bes78} is given by
\begin{equation*}
\theta_\eta(\nu):=\abs{\det \left(  g_\nu\left(\frac{\partial}{\partial\psi_i}(\nu),\frac{\partial}{\partial\psi_j}(\nu) \right) \right)_{i,j=1}^p}^{1/2},
\end{equation*}
where $g_\nu(\frac{\partial}{\partial\psi_i}(\nu),\frac{\partial}{\partial\psi_j}(\nu))$ denotes the metric $g$ in normal coordinates at the point $\exp_\eta^{-1}\nu$ (see e.g.~\cite[p.24]{Sak96}). Note that this function is only defined for points $\nu\in U$ such that $d_g(\eta,\nu)<\operatorname{inj}_gM$. Since $M$ is smooth, $\theta_\eta$ is continuous on $M$.

\section{Kernel density estimator of the mark distribution} \label{sectKernelEstim} 
In this section, we introduce a kernel density estimator for the density of the mark distribution on \edit{an}
observation window $B'_n\subset\R^d$. More precisely, we consider a sequence $\{B'_n\}_{n\in\Nat}$ of bounded Borel sets of $\R^d$ \textit{growing in the 
\edit{Van} Hove sense (\edit{VH}-growing sequence)}. This means that
\begin{equation*}\label{eqn def vanHove1}
\lim_{n\to\infty}\abs{B'_n}=\infty\qquad\text{and}\qquad\lim_{n\to\infty}\frac{\abs{\partial B'_n\oplus B(o,r)}}{\abs{B'_n}}=0,
\end{equation*}
where $B(o,r)$ denotes the ball of radius $r>0$ centered at the origin $o$. Given a set $B\subset\R^d$, $\abs{B}$ will denote its $d-$dimensional Lebesgue measure, where $d$ is the ``correct'' dimension of $B$, i.e. the one for which $B$ is a $d-$set. In this particular case, $\abs{B'_n}$ is the $d$-dimensional volume of $B'_n$.

\subsection{The estimator}

Let $\Psi=\{ (Y_i,\xi_i)\}_{i\geq 1}$ be an homogeneous Poisson marked point process of intensity $\lambda>0$. We define the kernel density estimator
\begin{equation*}
\hat{f}_n(\eta):=\frac{1}{\lambda|B_n'|}\sum_{i\geq 1}\frac{\mathds{1}_{\{Y_i\in B_n'\}}}{b_n^{p}\theta_\eta(\xi_i)}K\left(\frac{d_g(\eta,\xi_i)}{b_n}\right).
\end{equation*}
This is an extension of the estimator given by Pelletier in~\cite{Pel05}. The sequence of bandwidths $\{b_n\}_{n\in\Nat}\subset\R$ satisfies

\noindent
(b1) $b_n<r_0$ $\forall\,n\in\Nat$, with $0<r_0<\operatorname{inj}_g M$ and $\inf\limits_{\eta\in B_M(z,r_0)}\theta_z(\eta)>0$ for any $z\in M$,
\edit{
(b2) $b_n\downarrow 0$,\hspace*{.2in}(b3) $\lim\limits_{n\to\infty}b_n^p  \abs{B'_n}=\infty$.
}

\bigskip
The kernel $K\colon\R_+\to\R$ is a bounded nonnegative function satisfying

\noindent
(K1) $\supp K =[0,1]$,\hspace*{.15in} (K2) $\int_{\edit{B(o,1)}}K(\norm{x})dx=1$, 

\noindent
(K3) $0<\int_{\edit{B(o,1)}}K(\norm{x})\norm{x}^2 dx=:K_2<\infty$,\hspace*{.15in} (K4) $\sup_{r\geq 0} K(r)=:K_0<\infty$,

\noindent
(K5) $\int_{\edit{B(o,1)}}K(\norm{x})x\,dx=o$.

\bigskip

We further assume that

\noindent
(f1) $f \in L^2(M)$, i.e. $\norm{f}^2_2:=\int_M\abs{f(\eta)}^2 d\upsilon_g(\eta)<\infty,$
(f2) $f$ is twice continuously differentiable.
Property (f2) in particular means that $f$ has bounded Hessian on any normal neighborhood $U\subset M$, i.e. $\exists~C_2>0$ such that $\norm{D^2f}\leq C_2$.

\medskip

Assumptions on the kernel are standard when dealing with nonparametric density estimation~\cite{Pel05,Tsy09}. For the ease of notation, we will usually write 
\begin{equation*}
F_n(\eta,\xi):=\frac{1}{b_n^{p}\theta_\eta(\xi)}K\left(\frac{d_g(\eta,\xi)}{b_n}\right),\qquad\eta,\xi\in M.
\end{equation*}
In case the observation window \edit{$B'_n$} needs to be explicitly indicated in the notation, we will write $\fno$ instead of $\hat{f}_n$.

\subsection{Consistency}
In this section, we prove $L^2$ and almost sure consistency of $\hat{f}_n$. In what follows, $\omega_p$ will denote the 
\edit{volume} of the unit ball in $\R^p$ and we will write $x\cdot y$ for the Euclidean scalar product of any two vectors $x,y\in\R^p$. 

Note that in the classical (Euclidean) setting one could shorten proofs by applying Fourier methods~\cite{Tsy09}. However, in the general case of manifolds, this approach does not seem to be possible.

\begin{thm}\label{thm L2 conv norm}
Under the assumptions $(b1)-(b3)$, $(K1)-(K5)$, $(f1)$ and $(f2)$ we have that
\begin{equation*}
\Exp[\Vert\hat{f}_n- f\Vert^2_2]\leq\frac{C_\theta\omega_p K^2_0}{\lambda \abs{B'_n}b_n^{p}}+b_n^4 C_2^2 K_2^2\upsilon_g(M),
\end{equation*}
where $C_\theta:=\sup\limits_{z\in M}\sup\limits_{\eta\in B_M(z,r_0)}\theta_z(\eta)^{-1}$.
\end{thm}
\begin{cor}
Under the above assumptions, it follows directly from Theorem~\ref{thm L2 conv norm} that $\hat{f}_n$ is an $L^2$-consistent estimator of $f$, i.e. $\Exp[\Vert\hat{f}_n-f\Vert_2^2]\xrightarrow{n\to\infty}0$.
\end{cor}
\begin{cor}\label{thm L2 conv}
Under the assumptions of Theorem~\edit{\ref{thm L2 conv norm}} it holds that
\begin{equation*}
\Exp[|\hat{f}_n(\xio)- f(\xio)|^2]\xrightarrow{n\to\infty}0.
\end{equation*}
\end{cor}
In order to prove these results, we establish some auxiliary lemmata.
\begin{lemma}\label{lem: IntK is 1}
For each $\eta\in M$ and $n\in\Nat$,
\begin{equation}\label{eqn intK}
\int_{B_M(\eta,b_n)}\frac{1}{b_n^{p}\theta_\eta(z)}K\left(\frac{d_g(\eta,z)}{b_n}\right)d\upsilon_g(z)=1.
\end{equation}
\end{lemma}
\begin{proof}
Consider the exponential chart $(U,\psi)$ of $(M,g)$ introduced in Section~\ref{subsec space of marks} and set $z:=\exp_\eta (x)$, $B(0,b_n):=\exp_\eta B_M(\eta,b_n)$. Note that by definition (see~\cite[p.65]{Sak96} for details) the Jacobian of the transformation $\norm{g(x)}^{1/2}$ coincides with $\theta_\eta(\exp_\eta (x))$. The integral in~\eqref{eqn intK} thus becomes
\begin{equation*}
\int_{B(0,b_n)}\frac{1}{b_n^{p}\theta_\eta(\exp_\eta (x))}K\left(\frac{\norm{x}}{b_n}\right)\norm{g(x)}^{1/2}dx=\int_{B(0,1)}K\left(\norm{y}\right)\,dy =1.
\end{equation*}
\end{proof}

The calculations in the proof of this lemma lead to the useful equality
\begin{equation}\label{integ of B_M}
\int_{B_M(\eta,b_n)}\frac{1}{\theta_\eta(z)}d\upsilon_g(z)=\int_{B(0,b_n)}\frac{\norm{g(x)}^{1/2}}{\theta_\eta(\exp_\eta (x))}dx=\int_{B(0,b_n)}dx=b_n^p\omega_p.
\end{equation}
We give next an asymptotic bound for the bias of $\fn$.
\begin{lemma}\label{lemma Bias f}
For any $\eta\in \supp f$ and $n\in\Nat$
\begin{equation*}
\operatorname{Bias}\hat{f}_n( \eta ):=|\Exp[\hat{f}\edit{_n}(\eta)]-f(\eta)| \leq b_n^2 C_2K_2.
\end{equation*}
\end{lemma}
\begin{proof}
Let $\eta\in\supp f$. By the Campbell theorem,
\begin{equation}\label{eq Efn vs EFn}
\Exp[\hat{f}_n(\eta)]=\int_MF_n(\eta,z)f(z)\,d\upsilon_g(z)=\Exp[F_n(\eta,\xio)].
\end{equation}
Due to Lemma~\ref{lem: IntK is 1} and (K2) we have that
\begin{equation*}
\abs{\Exp\left[F_n(\eta,\xio)\right]-f(\eta)}=\abs{\int_{B_M(\eta,b_n)}\frac{1}{b_n^{p}\theta_\eta(z)}K\left(\frac{d_g(\eta,z)}{b_n}\right)\big(f(z)-f(\eta)\big)\,d\upsilon_g(z)}.
\end{equation*}
Consider now a normal neighborhood $\eta\in U\subset M$ and a point $x=(x^1,\ldots,x^{p})\in \T_\eta M$ in normal coordinates, i.e. $z=\exp_\eta(x)$. Further, define $\tilde{f}:=f\circ\exp_\eta$. The Taylor expansion of $f(z)$ around $\eta$ in normal coordinates is 
\[
f(z)=\tilde{f}(x)=\tilde{f}(0)+\nabla \tilde{f}(0)\mathbf{\cdot} x+ R_2(0,x),
\]
where $R_2(0,x)=O(x^TD^2\tilde{f}(0)x)$ is the second order 
\edit{remainder}. From assumption (f2) we have that $\abs{R_2(0,x)}\leq C_2\norm{x}^2$ for all $x\in B(0, b_n)$, hence passing to the exponential chart as in the proof of Lemma~\ref{lem: IntK is 1} yields
\begin{align}
&\abs{\int_{B_M(\eta,b_n)}\frac{1}{b_n^{p}\theta_\eta(z)}K\left(\frac{d_g(\eta,z)}{b_n}\right)\big(f(z)-f(\eta)\big)d\upsilon_g(z)}\nonumber\\
=\,&\abs{\int_{B(0,b_n)}\frac{1}{b_n^{p}}\frac{1}{\theta_\eta(\exp_\eta(x))}K\left(\frac{\norm{x}}{b_n}\right)\big(\tilde{f}(x)-\tilde{f}(0)\big)\norm{g(x)}^{1/2}dx}\label{eq EF-f change var}\\
=\,&\abs{\int_{B(0,b_n)}\frac{1}{b_n^{p}}\frac{1}{\theta_\eta(\exp_\eta(x))}K\left(\frac{\norm{x}}{b_n}\right)R_2(0,x)\norm{g(x)}^{1/2}dx}\label{eq3 Ef-f}\\
\leq\,&C_2\int_{B(0,b_n)}\frac{1}{b_n^{p}}K\left(\frac{\norm{x}}{b_n}\right)\norm{x}^2dx=C_2b_n^2K_2.\nonumber
\end{align}
Equality~\eqref{eq3 Ef-f} follows from~(K5) because
\begin{align*}
&\int_{B(0,b_n)}\frac{1}{b_n^{p}}K\left(\frac{\norm{x}}{b_n}\right)\nabla\tilde{f}(0)\cdot x\,dx=\sum_{i=1}^d\int_{B(0,b_n)}\frac{1}{b_n^{p}}K\left(\frac{\norm{x}}{b_n}\right)\nabla\tilde{f}(0)_ix^i\,dx\\
&=\sum_{i=1}^d\nabla\tilde{f}(0)_i\int_{B(0,b_n)}\frac{1}{b_n^{p}}K\left(\frac{\norm{x}}{b_n}\right)x^i\,dx=\nabla\tilde{f}(0)\cdot\underbrace{\int_{B(0,b_n)}\frac{1}{b_n^{p}}K\left(\frac{\norm{x}}{b_n}\right)x\,dx}_{=o}\edit{=0}.
\end{align*}
\end{proof}

\begin{lemma}\label{lemma Var}
For any $n\in\Nat$,
\begin{equation*}
\int_{M}\Exp[F_n^2(\eta, \xio)]\,d\upsilon_g(\eta)\leq\frac{C_\theta\,\omega_p K_0^2}{b_n^{p}},
\end{equation*}
with $C_\theta$ as in Theorem~\ref{thm L2 conv norm}.
\end{lemma}
\begin{proof}
Applying Fubini's theorem we write 
\begin{equation}\label{eqn lemma Var I}
\int_{M}\Exp[F_n^2(\eta, \xio)]\,d\upsilon_g(\eta)=\int_{M}I(z)f(z)\,d\upsilon_g(z),
\end{equation}
where
\[
I(z)=\int_{B_M(z,b_n)}\frac{1}{b_n^{2p}\theta^2_z(\eta)}K^2\left(\frac{d_g(\eta
,z)}{b_n}\right)d\upsilon_g(\eta).
\]
Let us define $C_\theta(z):=\sup\limits_{\eta\in B_M(z,r_0)}\theta_z(\eta)^{-1}$, which is finite because of (b1). By assumption (K4) and~\eqref{integ of B_M},
\begin{equation*}
I(z)\leq \frac{C_\theta(z)K^2_0}{b_n^{p}}\int_{B_M(z,b_n)}\frac{1}{b_n^{p}\theta_z(\eta)}d\upsilon_g(\eta)
=\frac{C_\theta(z)\omega_pK^2_0}{b_n^{p}}.
\end{equation*}
Plugging this estimate into~\eqref{eqn lemma Var I} finishes the proof.
\end{proof}

We proceed to prove Theorem~\ref{thm L2 conv norm}.

\begin{proof}[Proof of Theorem~\ref{thm L2 conv norm}]
By Fubini's theorem, 
\[
\Exp[\Vert\hat{f}_n- f\Vert^2_2]=\int_M \Exp[ |\hat{f}_n(\eta)-f(\eta)|^2]d\upsilon_g(\eta)=: \int_{M} J(\eta)  d\upsilon_g(\eta).
\]
Note that $J(\eta)=\Var(\edit{\hat{f}_n}(\eta))+\big(\operatorname{Bias}\edit{\hat{f}_n}(\eta)\big)^2$. In view of~\eqref{eq Efn vs EFn} and the Campbell theorem we get
\begin{align}
&\Var(\fn(\eta))= \Exp[\fn^2(\eta)]-(\Exp[\fn(\eta)])^2\nonumber\\
&=\frac{1}{\lambda^2|B_n'|^2}\Exp\Big[\sum_{i\geq 1}\mathds{1}_{\{Y_i\in B_n'\}}F^2_n(\eta,\xi_i)\Big]+\frac{1}{\lambda^2|B_n'|^2}\Exp\Big[\sideset{}{^{\neq}}\sum_{i,j\geq 1}\mathds{1}_{\{Y_i,Y_j\in B'_n\}}F_n(\eta,\xi_i)F_n(\eta,\xi_j)\Big]\nonumber\\
&-\Exp[F_n(\eta,\xio)]^2=\frac{1}{\lambda|B_n'|}\Exp[F^2_n(\eta,\xio)]+\frac{\alpha^{(2)}(B_n'\times B_n')}{\lambda^2|B'_n|^2}\Exp[F_n(\eta,\xio)]^2-\Exp[F_n(\eta,\xio)]^2\nonumber\\
&=\frac{1}{\lambda|B_n'|}\Exp[F^2_n(\eta,\xio)].\label{eq bound var f_n}
\end{align}
Here, $\alpha^{(2)}(\cdot)$ denotes the 2nd-order factorial moment measure of the Poisson point process $\Pi:=\{Y_i\}_{i\geq 1}$. We refer to~\cite[Chapter 1]{SKM87} for further definitions and formulas related to this measure in the Poisson case. Corollary~\ref{lemma Bias f} and Lemma~\ref{lemma Var} yield the existence of constants $C_\theta, C_2>0$ such that
\begin{equation*}
\Exp[\Vert\hat{f}_n- f\Vert^2_2]\leq \frac{C_\theta\omega_p K^2_0}{\lambda \abs{B'_n}b_n^{p}}+b_n^4 C_2^2 K_2^2\upsilon_g(M).
\end{equation*}
\end{proof} 
Analogous arguments show the $L^2$-convergence of $\fn(\xio)$ to $f(\xio)$.
\begin{proof}[Proof of Corollary~\ref{thm L2 conv}]
Passing to normal coordinates as in~\eqref{eq EF-f change var} and~\eqref{eq3 Ef-f} and setting $\tilde{f}:=f\circ\exp_\eta$ \edit{lead to}
\begin{equation}\label{eq EF_n approx f}
\Exp[F_n(\eta,\xio)]=\int_{B(0,b_n)}\frac{1}{b_n^p}K\left(\frac{\norm{x}}{b_n}\right)\tilde{f}(x)dx=(1+o(1))f(\eta).
\end{equation}
From the proof of Lemma~\ref{lemma Var} we thus obtain 
\begin{equation}\label{eq bound F_n2}
\Exp[F_n^2(\eta,\xio)]\leq\frac{K_0C_\theta(\eta)}{b_n^p}\Exp[F_n(\eta,\xio)]\leq\frac{2K_0C_\theta(\eta)}{b_n^p}f(\eta)
\end{equation}
for any $\eta\in\supp f$. In view of~\eqref{eq bound var f_n} and Lemma~\ref{lemma Bias f}, this yields
\begin{equation*}
\Exp[|\fn(\xio)-f(\xio)|^2]\leq\frac{2C_\theta K_0\norm{f}_2^2}{\lambda b^p_n|B_n'|}+b_n^4C_2^2K_2^2,
\end{equation*}
which tends to zero as $n\to\infty$.
\end{proof}
\begin{rem}\label{rem optimal b_n density}
The problem of finding an optimal sequence of bandwidths $\{b_n\}_{n\in\Nat}$ can be understood as a special case of regularization~\cite{Sch07} and 
the bound of the estimation error given in Theorem~\ref{thm L2 conv norm} \edit{can be used} in order to find it. 
For any fixed $n\in\Nat$, the optimal bandwidth will be 
$\operatorname{argmin}_{\edit{b_n}} \Exp[\Vert\hat{f}_n- f\Vert^2_2]$. \edit{Applying} Theorem~\ref{thm L2 conv norm}, we can approximate the order of magnitude of this optimal $b_n$ by minimizing the upper bound of the mean square error $e(b_n):=\frac{C_\theta\omega_p K^2_0}{\lambda \abs{B'_n}b_n^{p}}+b_n^4 C_2^2 K_2^2\upsilon_g(M)$. A simple calculation leads to the unique minimum point $b_{n,opt}=\left(\frac{pC_\theta\omega_pK_0^2}{4C_2^2K_2^2\upsilon_g(M)\lambda\abs{B'_n}}\right)^{\frac{1}{p+4}}$.
Note that 
$b_{n,opt}\downarrow 0$ and $b^p_{n,opt}\abs{B'_n}\to\infty$ as $n\to\infty$.
\end{rem}
We finish this section by proving that if the observation window $B_n'$ is large enough, then the previous bounds provide the almost surely consistency of $\hat{f}_n$.
\begin{thm}\label{thm as conv KDE}
Under the assumptions of Theorem~\ref{thm L2 conv norm}, choosing $b_n=o(n^{-\frac{1+\delta}{4}})$ and $B_n'$ such that
\edit{ $b_n^p|B'_n|>n^{1+\delta}$} for some $\delta>0$,
\begin{equation*}
|\hat{f}_n(\eta)-f(\eta)|\xrightarrow{n\to\infty}0\qquad a.s.
\end{equation*}
for any $\eta\in M$ such that $f(\eta)<\infty$.
\end{thm}
\begin{proof}
%
For each $\varepsilon>0$, Chebyshev's inequality and the bounds used in the proof of Corollary~\ref{thm L2 conv} yield
\begin{align*}
\Pro(|\hat{f}_n(\eta)-f(\eta)|>\varepsilon)&\leq\frac{\Exp[|\hat{f}_n(\eta)-f(\eta)|^2]}{\varepsilon^2}\leq\frac{2C_\theta K_0f(\eta)}{\varepsilon^2\lambda b_n^p|B_n'|}+\frac{b^4_nC^2_2K^2_2}{\varepsilon^2}.
\end{align*}
Due to the choice of $b_n$ we have $b_n^p|B_n'|>n^{1+\delta}$ and $b_n^4<n^{-(1+\delta)}$, hence
\begin{align*}
\sum_{n=1}^\infty\Pro(|\hat{f}_n(\eta)-f(\eta)|>\varepsilon)\leq c_1f(\eta)\sum_{n=1}^\infty\frac{1}{n^{1+\delta}}<\infty
\end{align*}
for some $c_1<\infty$. The almost sure convergence follows from Borel-Cantelli's lemma.
\end{proof}
\section{Entropy estimator}\label{sectEntrEst}
As already mentioned in the introduction, we measure the diversity of the distribution of interest by analyzing its Kolmogorov entropy defined as
\begin{equation*}\label{eqn def entropy}
\e_f:=-\int_{M} f(\eta) \log f(\eta)\,d\upsilon_g(\eta),
\end{equation*}
where $f$ is the density of the distribution. This section is devoted to the construction of a consistent estimator for $\e_f$.
\subsection{Definition of the estimator and consistency}
For each $n\in\Nat$ we define
\begin{equation}\label{eq def EE}
\widehat{\e}_f(B_n):=-\frac{1}{\lambda|B_n|}\sum\limits_{i\geq 1}\mathds{1}_{\{Y_i\in B_n\}}\log \fnYi(\xi_i),
\end{equation}
where $B_n'+y$ denotes the translation of $B_n'$ by $y\in\R^d$ \edit{and $B'_n\subseteq B_n$}. 
\edit{The window $B'_n$ is introduced for the purpose of notation and it will become relevant when proving the CLT in Section 5. Throughout this section we have no restrictions on it and we can assume $B_n=B_n'$.}

From now on, we substitute the previous assumption (f1) by $f$ being continuous. Note that since $M$ is compact, the new (f1) in particular implies the former. With the additional assumptions for a typical mark $\xi_{\mathbf{0}}$,
\[
\text{(L1)}~\Exp\left[\log^2 f(\xi_\mathbf{0})\right]=:L_1<\infty~\qquad\text{and}\qquad
\text{(L2)}~\Exp\left[\left(\frac{\norm{\nabla f(\xi_{\mathbf{0}})}}{f(\xi_{\mathbf{0}})}\right)^2\right]=:L_2<\infty,
\]
we can prove $L^2$-consistency of the estimator.

\begin{thm}\label{thm L2 conv EE}
\edit{
For each $n\in\Nat$, let $\{B_n\}_{n\in\Nat}$ and $\{B'_n\}_{n\in\Nat}$ be sequences of VH-growing Borel sets satisfying $(b1)-(b3)$.} Further, assume that conditions $(K1)-(K5)$, $(f1)$, $(f2)$, $(L1)$ and $(L2)$ hold. Then,
\begin{equation*}\label{eqn L2 conv EE}
\Exp[\vert\widehat{\e}_f(B_n)-\e_f\vert^2]\leq 3\left(\frac{8K_0C_\theta\upsilon_g(M)}{\lambda^2|B_n||B_n'|b_n^p}+\frac{4}{\lambda^2|B_n'|}+32b_n^2L_2+\frac{L_1}{\lambda |B_n|}\right)
\end{equation*}
for sufficiently large $n\in\Nat$.
\end{thm}
\begin{cor}\label{cor L2 conv EE}
Under the above assumptions, it follows directly from Theorem~\ref{thm L2 conv EE} that $\widehat{\e}_f(B_n)$ is an $L^2$-consistent estimator of $\e_f$, i.e. $\Exp[\vert\widehat{\e}_f(B_n)-\e_f\vert^2]\xrightarrow{n\to\infty}0$.
\end{cor}

\subsection{Proof of Theorem~\ref{thm L2 conv EE}}
We start by proving the following lemma assuming that all conditions of Theorem~\ref{thm L2 conv EE} are satisfied.
\begin{lemma}\label{lemma Ef_n2/f_n integ}
For sufficiently large $n\in\Nat$ it holds that
\begin{equation*}
\int_{\supp f}\frac{(\Exp[\fBn(\eta)]-f(\eta))^2}{f(\eta)}\,d\upsilon_g(\eta)\leq 4b_n^2L_2.
\end{equation*}
\end{lemma}
\begin{proof}
Recall from~\eqref{eq Efn vs EFn} that $\Exp[\fBn(\eta)]=\Exp[F_n(\eta,\xio)]$. Using normal coordinates analogously to~\eqref{eq EF-f change var} and~\eqref{eq3 Ef-f} with $\tilde{f}:=f\circ\exp_\eta$ we obtain
\begin{align*}
\abs{\Exp[F_n(\eta,\xio)]-f(\eta)}
&=\Big|\int_{B(0,b_n)}\frac{1}{b_n^{p}}K\left(\frac{\norm{x}}{b_n}\right)x\cdot\int_0^1\nabla\tilde{f}(tx)\,dt\,dx\Big|\\
&\leq b_n\int_{B(0,1)}K\left(\norm{y}\right)\norm{y}\int_0^1\big\Vert\nabla\tilde{f}(tb_ny)\big\Vert\,dt\,dy.
\end{align*}
Since $b_n\downarrow0$, we have $\big\Vert\nabla\tilde{f}(tb_ny)\big\Vert=\big\Vert\nabla\tilde{f}(0)\big\Vert(1+o(1))$ for sufficiently large $n\in\Nat$ and in view of (K2), last expression can be bounded by
\begin{equation*}
2b_n\big\Vert\nabla\tilde{f}(0)\big\Vert\int_{B(0,1)}K\left(\norm{y}\right)\norm{y}dy
\leq 2b_n\big\Vert\nabla\tilde{f}(0)\big\Vert.
\end{equation*}
Hence, $|\Exp[\fBn(\eta)]-f(\eta)|\leq 2 b_n\norm{\nabla f(\eta)}$ for sufficiently large $n\in\Nat$ and (L2) yields
\begin{equation*}
\int_{\supp f}\frac{(\Exp[\fBn(\eta)]-f(\eta))^2}{f(\eta)}\,d\upsilon_g(\eta)
\leq 4b_n^2\int_{\supp f}\frac{\norm{\nabla f(\eta)}^2}{f(\eta)}\,d\upsilon_g(\eta)= 4b_n^2L_2.
\end{equation*}
\end{proof}
We now proceed to prove Theorem~\ref{thm L2 conv EE}. Based on~\cite{AL76}, we introduce the quantities
\begin{align*}
&L_n:=-\frac{1}{\lambda|B_n|}\sum_{i\geq 1}\mathds{1}_{\{Y_i\in B_n\}}\log\Exp[\fnYi(\xi_i)],\\
&M_n:=-\frac{1}{\lambda|B_n|}\sum_{i\geq 1}\mathds{1}_{\{Y_i\in B_n\}}\log f(\xi_i).
\end{align*}
Applying inequality $(a+b+c)^2\leq 3(a^2+b^2+c^2)$, $a,b,c\in\R$, leads to
\begin{equation*}
\Exp\big[\vert\widehat{\e}_f(B_n)-\e_f\vert^2\big]\leq 3\big(\underbrace{\Exp\big[\vert\widehat{\e}_f(B_n)-L_n\vert^2\big]}_{=:I_{1,n}}+\underbrace{\Exp\big[\abs{L_n-M_n}^2\big]}_{=:I_{2,n}}+\underbrace{\Exp\big[|M_n-\e_f|^2\big]}_{=:I_{3,n}}\big),
\end{equation*}
hence our aim is to compute an upper bound for $I_{i,n}$ and each $i=1,2,3$. First,
\begin{align*}
I_{1,n}&=\frac{1}{\lambda^2|B_n|^2}\Exp\big[\sum_{i\geq 1}\mathds{1}_{\{Y_i\in B_n\}}(\log\fnYi(\xi_i)-\log\Exp[\fnYi(\xi_i)])^2\big]\\
&+\frac{1}{\lambda^2|B_n|^2}\Exp\big[\sideset{}{^{\neq}}\sum_{i,j\geq 1}\mathds{1}_{\{Y_i,Y_j\in B_n\}}(\log\fnYi(\xi_i)-\log\Exp[\fnYi(\xi_i)])\times\\
&\times(\log\fnYj(\xi_j)-\log\Exp[\fnYj(\xi_j)])\big]=:J_1+J_2.
\end{align*}
On the one hand, notice that by definition, 
 \[
h(Y_i,\xi_i,T_{Y_i}\Psi-\delta_{(o,\xi_i)}):=\mathds{1}_{\{Y_i\in B_n\}}(\log\fnYi(\xi_i)-\log\Exp[\fnYi(\xi_i)])^2
\]
depends on $(Y_i,\xi_i)$ and $T_{Y_i}\Psi-\delta_{(o,\xi_i)}$. Since $\Psi$ is an independently marked Poisson MPP, the Campbell-Mecke type formula in~\cite[p.129]{Spo13} yields
\begin{align*}
\frac{1}{\lambda^2|B_n|^2}\Exp\big[\sum_{i\geq 1}h(Y_i,\xi_i,T_{Y_i}\Psi-\delta_{(o,\xi_i)})\big]
=\frac{1}{\lambda|B_n|^2}\int_{\R^d}\int_M\Exp_{P^{o!}_{\eta}}[h(y,\eta,\Psi)]f(\eta)d\upsilon_g(\eta)dy,
\end{align*} 
where $\Exp_{P^{!}_{(o,\eta)}}$ denotes expectation with respect to the reduced Palm distribution of $\Psi$. Again because $\Psi$ is an independently marked Poisson MPP, 
$P^{!}_{(o,\eta)}$ coincides with the distribution of $\Psi$ and we obtain
\begin{align*}
J_1&=\frac{1}{\lambda|B_n|^2}\int_{B_n}\int_M\Exp[(\log\fny(\eta)-\log\Exp[\fny(\eta)])^2]f(\eta)d\upsilon_g(\eta)dy.
\end{align*}
\edit{Notice that $\log x$ is a differentiable function, hence the mean value theorem yields}
\begin{equation}\label{eq mean value}
|\log x-\log z|=\frac{|x-z|}{|(1-\gamma)x+\gamma z|}\leq \frac{|x-z|}{\min\{x,z\}},\qquad x,z> 0
\end{equation}
for some \edit{$\gamma\in(0,1)$}. \edit{Since} $\Psi$ is stationary and by assumption (f1) $f$ is continuous, $\fny(\eta)$ converges to $f(\eta)$ a.s. for any $y\in\R^d$ and $\eta\in M$ by Theorem~\ref{thm as conv KDE}. Furthermore, in view of~\eqref{eq EF_n approx f}, $\Exp[F_n(\eta,\xio)]=(1+o(1))f(\eta)$, hence for $n\in\Nat$ large enough 
\begin{equation}\label{eq min fny f}
\min\{\fny(\eta),\Exp[\fny(\eta)]\}\geq \frac{1}{2} f(\eta).
\end{equation}
Applying inequality~\eqref{eq mean value} with $x=\fny(\eta)$ and $z=\Exp[\fny(\eta)]=\Exp[F_n(\eta,\xio)]$ we obtain
\begin{align*}
J_1&\leq \frac{4}{\lambda|B_n|^2}\int_{B_n}\int_M\frac{\Exp[(\fny(\eta)-\Exp[\fny(\eta)])^2]}{f(\eta)^2}f(\eta)d\upsilon_g(\eta)dy.
\end{align*}
Due to~\eqref{eq bound var f_n} and~\eqref{eq bound F_n2},
\begin{equation}\label{eq bound J_1}
J_1\leq\frac{4}{\lambda|B_n|}\int_M\frac{\Exp[F_n^2(\eta,\xio)]}{f(\eta)^2\lambda|B_n'|}f(\eta)d\upsilon_g(\eta)dy\leq\frac{8K_0C_\theta\upsilon_g(M)}{\lambda^2|B_n||B'_n|b_n^p}.
\end{equation}

Analogously, each summand in $J_2$ can be expressed as a function $h$ depending of $(Y_i,\xi_i)$, $(Y_j,\xi_j)$ and  $T_{Y_i}\Psi-\delta_{(o,\xi_i)}-\delta_{(Y_j,\xi_j)}$. Hence, the Campbell-Mecke type formula in~\cite[p.129]{Spo13} in the independently marked Poisson case yields
\begin{align*}
J_2&=\Exp\Big[\sideset{}{^{\neq}}\sum_{i,j\geq 1}h(Y_i,\xi_i,Y_j,\xi_j,T_{Y_i}\Psi-\delta_{(o,\xi_i)}-\delta_{(Y_j,\xi_j)})\Big]\\
&=\lambda\int_{\R^d}\int_{\R^d}\int_{M^2}\Exp_{P_{\eta_1,\eta_2}^{o,y_2!}}[h(y_1,\eta_1,y_2,\eta_2,\Psi)]f(\eta_1)f(\eta_2)\,d\upsilon_g(\eta_2)d\upsilon_g(\eta_1)dy_1\,dy_2\\
&=\lambda\int_{\R^d}\int_{\R^d}\int_{M^2}\Exp[h(y_1,\eta_1,y_2,\eta_2,\Psi)]f(\eta_1)f(\eta_2)\,d\upsilon_g(\eta_2)d\upsilon_g(\eta_1)dy_1\,dy_2,
\end{align*}
where last inequality follows from the independent marking of the Poisson MPP. Applying again Theorem~\ref{thm as conv KDE},~\eqref{eq mean value} and~\eqref{eq EF_n approx f}, we obtain for $n\in\Nat$ large enough
\begin{align*}
J_2&\leq\frac{4}{\lambda|B_n|^2}\int_{(B_n\times M)^2}\frac{\Cov\big(\fnyI(\eta_1),\fnyII(\eta_2)\big)}{f(\eta_1)f(\eta_2)} f(\eta_1)f(\eta_2)\,d\upsilon_g(\eta_2)d\upsilon_g(\eta_1)dy_1\,dy_2.
\end{align*}
In view of~\eqref{eq Efn vs EFn} and the Campbell theorem,
\begin{align*}
&\Cov\big(\fnyI(\eta_1),\fnyII(\eta_2)\big)=\Exp[\fnyI(\eta_1)\fnyII(\eta_2)]-\Exp[\fnyI(\eta_1)]\Exp[\fnyII(\eta_2)]\\
&=\frac{1}{\lambda^2|B_n'|^2}\Exp\big[\sum_{i\geq 1}\mathds{1}_{\{Y_i\in(B_n'+y_1)\cap(B_n'+y_2)\}}F_n(\eta_1,\xi_i)F_n(\eta_2,\xi_i)\big]\\
&+\frac{1}{\lambda^2|B_n'|^2}\Exp\big[\sideset{}{^{\neq}}\sum_{i,j\geq 1}\mathds{1}_{\{Y_i\in B_n'+y_1\}}\mathds{1}_{\{Y_j\in B_n'+y_2\}}F_n(\eta_1,\xi_i)F_n(\eta_2,\xi_j)\big]\\
&-\Exp[F_n(\eta_1,\xio)]\Exp[F_n(\eta_2,\xio)]\\
&=\frac{|(B_n'+y_1)\cap (B_n'+y_2)|}{\lambda|B_n'|^2}\Exp[F_n(\eta_1,\xio)F_n(\eta_2,\xio)]\leq\frac{1}{\lambda|B_n'|}\Exp[F_n(\eta_1,\xio)F_n(\eta_2,\xio)].
\end{align*}
Fubini's theorem and Lemma~\ref{lem: IntK is 1} yield
\begin{equation*}
J_2\leq\frac{4}{\lambda^2|B_n'|}\int_{M^2}\Exp[F_n(\eta_1,\xio)F_n(\eta_2,\xio)]d\upsilon_g(\eta_1)d\upsilon_g(\eta_2)=\frac{4}{\lambda^2|B_n'|},
\end{equation*}
which together with~\eqref{eq bound J_1} leads to $I_{1,n}\leq \frac{8K_0C_\theta\upsilon_g(M)}{\lambda^2|B_n||B'_n|b_n^p}+\frac{4}{\lambda^2|B_n'|}$. Secondly, due to the stationarity of $\Psi$ and the Campbell theorem we have for large $n\in\Nat$
\begin{align*}
&I_{2,n}=\frac{1}{\lambda^2|B_n|^2}\Exp\big[\sum_{i\geq 1}\mathds{1}_{\{Y_i\in B_n\}}(\log\Exp[\fnYi(\xi_i)]-\log f(\xi_i))^2\big]\\
&+\frac{1}{\lambda^2|B_n|^2}\Exp\big[\sideset{}{^{\neq}}\sum_{i,j\geq 1}\mathds{1}_{\{Y_i,Y_j\in B_n\}}(\log\Exp[\fnYi(\xi_i)]-\log f(\xi_i))(\log\Exp[\fnYj(\xi_j)]-\log f(\xi_j))\big]\\
&=\frac{1}{\lambda|B_n|}\Exp[(\log\Exp[\fno(\xio)]-\log f(\xio))^2]+(\Exp[\log \Exp[\fno(\xio)]-\log f(\xio)])^2\\
&\leq 2\Exp[(\log\Exp[\fno(\xio)]-\log f(\xio))^2].
\end{align*}
On the other hand, by~\eqref{eq min fny f} and Lemma~\ref{lemma Ef_n2/f_n integ} we get
\begin{align*}
&\Exp[(\log\Exp[\fno(\xio)]-\log f(\xio))^2]
&\leq 4\int_{\supp f}\frac{(\Exp[\fno(\eta)]-f(\eta))^2}{f(\eta)}d\upsilon(\eta)\leq 16b_n^2L_2,
\end{align*}
so that $I_{2,n}\leq 32b_n^2L_2$.

Finally, note that $\e_f=-\Exp[\log f(\xio)]$. Applying once again the Campbell theorem we obtain
\begin{align*}
I_{3,n}&=\frac{1}{\lambda^2|B_n|^2}\bigg(\Exp\big[\sum_{i\geq 1}\mathds{1}_{\{Y_i\in B_n\}}\log^2f(\xi_i)\big]+\Exp\big[\sideset{}{^{\neq}}\sum_{i,j\geq 1}\mathds{1}_{\{Y_i,Y_j\in B_n\}}\log f(\xi_i)\log f(\xi_j)\big]\bigg)\\
&+\frac{2}{\lambda|B_n|}\Exp\big[\sum_{i\geq 1}\mathds{1}_{\{Y_i\in B_n\}}\log f(\xi_i)\big]\e_f+\e_f^2\\
&=\frac{1}{\lambda|B_n|}\Exp[\log^2f(\xio)]+\big(\Exp[\log f(\xio)]\big)^2+2\Exp[\log f(\xio)]\e_f+\e_f^2\\
&=\frac{1}{\lambda|B_n|}\Exp[\log^2f(\xio)]=\frac{L_1}{\lambda|B_n|}.
\end{align*}
\begin{rem}
The proof of Theorem~\ref{thm L2 conv EE} gives an explicit bound of the error that can be used to find an optimal sequence of bandwidths. In this case analogous calculations to Remark~\ref{rem optimal b_n density} lead to $b_{\edit{n},opt}=\Big(\frac{pK_0C_\theta\upsilon_g(M)}{4L_2\lambda^2|B_n||B'_n|}\Big)^{\frac{1}{p+2}}$.
\end{rem}
\section{Central limit theorem for entropy} \label{sectEntrCLT}
\edit{If the window $B'_n$ satisfies that $B_n'\subset B_n$, and $m_n$ is the diameter of $B'_n$, }the estimator $\widehat{\e}_f(B_n)$ can be seen as a normalized random sum of elements of a stationary \emph{$m_n-$dependent random field}. 
\edit{In this section, we present} a CLT for a modified version of the original estimator.

Let us start by fixing some notation. In general, we use uppercase for coordinates and lowercase for enumerating elements. For $\K\in\{\Nat,\Z,\R\}$, any $j\in\K^d$ will therefore be written as $j=(j^1,\ldots, j^d)$, while $j_1,j_2,\ldots$ will denote a sequence in $\K^d$. Moreover, we write $\mathbf{t}=(t,\ldots,t)\in\K^d$ for any $t\in\K$. We set $C_y:=\times_{k=1}^d [0,y^k)$ for any $y\in\R_+^d$ and $V_j:=C_j\cap\Nat^d$ for $j\in\Nat^d$. In particular, $C_{\mathbf{t}}=[0,t)^d$ for $t\in\R_+$.

A random field $\{X_j,\,j\in\K^d\}$ is said to be $m-$\emph{dependent} for some \edit{$m>0$} if for any finite sets $I,J\subset\K^d$ the random vectors $(X_i)_{i\in I}$ and $(X_j)_{j\in J}$ are independent whenever $\norm{i-j}_\infty>m$ for all $i\in I$ and $j\in J$.

In stochastic geometry, $m$-dependent random fields often appear in connection with models based on independently marked point processes. A CLT for sums of $m$-dependent random fields was first investigated by Ros\'en~\cite{Ros69} and improved by Heinrich~\cite{Hei88b}. These results have been extended in the last years to weaker dependence structures
\edit{(see~\cite{CS04,Spo13} and references therein).}
\subsection{Theoretical results}
\new{Our CLT is based on the following result by Chen and Shao~\cite{CS04} for deterministic sums of $m$-dependent random fields.
\begin{thm}\cite[Theorem 2.6]{CS04}\label{thm Chen Shao}
Let $\{X_i\}_{i\in I}$, $I\subseteq\Nat^d$, be a centered $m$-dependent random field such that $\Exp[|X_i|^q]<\infty$ for some $2<q\leq 3$ and any $i\in I$. Then,
\[
\sup_{x\in\R}|F(x)-\Phi(x)|\leq 75(10m+1)^{(q-1)d}\big(\Var \sum_{i\in I}X_i\big)^{-q/2}\sum_{i\in I}\Exp[|X_i|^q],
\]
where $F$ is the distribution function of $\big(\Var \sum_{i\in I}X_i\big)^{-1/2}\sum_{i\in I}X_i$.
\end{thm}
}

We give an extension of this theorem to random sums of stationary \edit{$m_n$}-dependent random fields indexed in $\R^d_+$. For simplicity, we assume that our observation windows are cubic, i.e. $B_n:=C_{\mathbf{p_n}}$ with $p_n\to\infty$ as $n\to\infty$.
\new{
\begin{cor}\label{cor CLT mdep random sum}
Let $\{X_{n,y},\,y\in B_n\}_{n\in\Nat}$ be a sequence of stationary centered $m_n$-dependent random fields and let $\Pi$ be a stationary Poisson point process on $\R^d_+$. Assume that 
\begin{equation}\label{eq A}
\sup_{n\in\Nat}\Exp\Big[\Big|\sum_{y\in\Pi\cap C_{\mathbf{1}}}X_{n,y}\Big|^q\Big]<\infty\tag{A}
\end{equation}
for some $2<q\leq 3$. Then,
\begin{equation*}
\sup_{x\in\R}|F_n(x)-\Phi(x)|\leq 75(10m_n+11)^{(q-1)d}|B_n|\sigma_n^{-q}\Exp\Big[\Big|\sum_{y\in\Pi\cap C_{\mathbf{1}}}X_{n,y}\Big|^q\Big],
\end{equation*}
where $\sigma_n^2=\Var\sum\limits_{y\in\Pi\cap B_n }X_{n,y}$ and $F_n$ is the distribution function of $\sum\limits_{y\in\Pi\cap B_n }X_{n,y}/\sigma_n$.
\end{cor}
}
\begin{proof}
For each $j\in\Nat^d$ and $n\in\Nat$, define $Z_{n,j}:=\sum_{y\in\Pi\cap(C_{\mathbf{1}}+j)}X_{n,y}$. Obviously, $\{Z_{n,j}\}_{j\in V_{\mathbf{p_n}}}$ is a stationary centered $(m_n+1)$-dependent random field with $\sup_{n\in\Nat}\Exp[|Z_{n,j}|^q]<\infty$ for any $j\in V_{\mathbf{p_n}}$ and $2<q\leq 3$. Hence, Theorem~\ref{thm Chen Shao} with $I=V_{\mathbf{p_n}}$ yields the stated bound.
\end{proof}
\vspace*{-.25in}
\new{\begin{rem}\label{rem indep PP}
Note that Corollary~\ref{cor CLT mdep random sum} does not require independence between the random fields $\{X_{n,y}\}_{y\in B_n}$ and the point process $\Pi$. If independence is provided, the Campbell theorem together with the generalized Cauchy-Schwartz inequality and the stationarity of $\Pi$ lead to
\begin{align*}
&\Exp\Big[\Big|\sum_{y\in\Pi\cap C_1}X_{n,y}\Big|^3\Big]
\leq\lambda\int_{C_1}\Exp[|X_{n,y}|^3]dy+\lambda\int_{C_1^2}\Exp[X_{n,y_1}^2 |X_{n,y_2}|]\,\alpha^{(2)}(dy_1,dy_2)\\
&+\lambda\int_{C_1^3}\Exp[|X_{n,y_1} X_{n,y_2}X_{n,y_3}|]\,\alpha^{(3)}(dy_1,dy_2,dy_3)\\
&\leq\lambda\Exp[|X_{n,\mathbf{0}}|^3](1+\alpha^{(2)}(C_1^2)+\alpha^{(3)}(C_1^3))=\lambda\Exp[|X_{n,\mathbf{0}}|^3](1+\lambda^2+\lambda^3),
\end{align*}
where $\lambda>0$ is the intensity of $\Pi$ and $\alpha^{(k)}$, $k=2,3$, denotes the $k$-th order factorial moment measure of $\Pi$ (see~\cite[Chapter 1]{SKM87} for explicit formulas in the Poisson case). Thus, we may substitute condition~\eqref{eq A} by
\begin{equation}\label{eq A prime}
\sup_{n\in\Nat}\Exp[|X_{n,\mathbf{0}}|^3]<\infty\tag{A'}
\end{equation}
and obtain Corollary~\ref{cor CLT mdep random sum} in the case $q=3$.
\end{rem}
}
Before applying \edit{Corollary~\ref{cor CLT mdep random sum} and Remark~\ref{rem indep PP}} to our entropy estimator, we want to investigate under which conditions the limiting variance exists. The following theorem is an extension of~\cite[Theorem 1.8, p.175]{BS07} to random sums of 
\edit{wide-sense} stationary random fields indexed in $\R^d$.

\begin{thm}\label{thm lim var}
Let $\{X_{n,y},\,y\in\R^d\}_{n\in\Nat}$ be a sequence of wide-sense stationary measurable centered random fields and let $\Pi$ be a homogeneous Poisson point process of intensity $\lambda>0$ independent of $\{X_{n,y},\,y\in\R^d\}$. Assume that
\begin{equation}\label{A1 thm existence var}
\lim_{p\to\infty}\limsup_{n\to\infty}\int_{\R^d\setminus (-p,p)^d}|\Cov(\Xno,\Xny)|\,dy=0,
\end{equation}
and
\begin{equation}\label{A2 thm existence var}
\sup_{n\in\Nat}\int_{\R^d}|\Cov(\Xno,\Xny)|\,dy<\infty.
\end{equation}
If the limit 
\begin{equation*}
\sigma^2:=\lim_{n\to\infty}\Big(\lambda\Exp[\Xno^2]+\lambda^2\int_{\R^d}\Cov(\Xno,\Xny)\,dy\Big)
\end{equation*}
exists and is positive, then
\begin{equation}\label{lim var}
\frac{1}{|U_n|}\Var\Big(\sum_{y\in\Pi\cap U_n}X_{n,y}\Big)\xrightarrow{n\to\infty}\sigma^2
\end{equation}
for any \edit{VH-growing }sequence $\{U_n\}_{n\in\Nat}\edit{\subseteq\R^d}$.
\end{thm}
\begin{proof}
Since $\Pi$ is a Poisson point process independent of $\{X_{n,y}\}_{y\in\R^d}$, it follows from the Campbell theorem and the wide-sense stationarity that 
\begin{align*}
\Var\Big(\sum_{y\in\Pi\cap U_n}\Xny\Big)
&=\lambda|U_n|\Exp[\Xno^2]+\lambda^2|U_n|\int_{\R^d}\Cov(\Xno,\Xny)\,dy\\
&-\lambda^2\int_{U_n}\int_{U_n^c}\Cov(\XnI,\XnII)\,dy_1dy_2.
\end{align*}
Following the proof of~\cite[Theorem 1.8]{BS07}, let $p>0$ be arbitrary and set $G_n:=U_n\cap(\partial U_n)_p$, $W_n:=U_n\setminus G_n$, where $(\partial U_n)_p:=\partial U_n\oplus B(\mathbf{0},p)$ denotes the $p-$neighborhood of $\partial U_n\subset\R^d$. From the previous calculation we have
\begin{align*}
&\lambda|U_n|\Exp[\Xno^2]+\lambda^2|U_n|\int_{\R^d}\Cov(\Xno,\Xny)\,dy-\Var\Big(\sum_{y\in\Pi\cap U_n}\Xny\Big)\\
&=\lambda^2\int_{G_n}\int_{U_n^c}\Cov(\XnI,\XnII)\,dy_1dy_2+\lambda^2\int_{W_n}\int_{U_n^c}\Cov(\XnI,\XnII)\,dy_1dy_2\\
&=:R_{n,1}+R_{n,2}.
\end{align*}
On the one hand, $|G_n|\leq|(\partial U_n)_p|$ and since $\{U_n\}_{n\in\Nat}$ is \edit{VH}-growing, assumption~\eqref{A2 thm existence var} yields
\begin{equation*}
\frac{|R_{n,1}|}{|U_n|}\leq\frac{|(\partial U_n|)_p}{|U_n|}\lambda^2\int_{\R^d}|\Cov(\Xno,\Xny)|\,dy\xrightarrow{n\to\infty}0.
\end{equation*}
On the other hand, $\operatorname{dist}(W_n,U_n^c)\geq p$ and $|W_n|\leq |U_n|$, hence
\begin{equation*}
\frac{|R_{n,2}|}{|U_n|}\leq\frac{|W_n|}{|U_n|}\lambda^2\int_{\R^d\setminus (-p,p)^d}|\Cov(\Xno,\Xny)|\,dy\leq\lambda^2\int_{\R^d\setminus (-p,p)^d}|\Cov(\Xno,\Xny)|\,dy
\end{equation*}
and in view of assumption~\eqref{A1 thm existence var} the convergence in~\eqref{lim var} is established.
\end{proof}

The same holds under weaker assumptions if the random fields $\{X_{n,y},\,y\in\R^d\}_{n\in\Nat}$ are $m_n-$dependent.

\begin{cor}\label{cor lim var mdep}
Let $\{X_{n,y},\,y\in\R^d\}_{n\in\Nat}$ be a sequence of wide-sense stationary measurable centered $m_n$-dependent random fields and let $\Pi$ be a homogeneous Poisson point process of intensity $\lambda>0$ independent of $\{X_{n,y},\,y\in\R^d\}_{n\in\Nat}$. Assume that
\begin{equation}\label{A1 thm existence var mdep case}
\sup_{n\in\Nat}\int_{\R^d}|\Cov(\Xno,\Xny)|\,dy<\infty.
\end{equation}
If the limit 
\begin{equation*}
\sigma^2:=\lim_{n\to\infty}\Big(\lambda\Exp[\Xno^2]+\lambda^2\int_{\R^d}\Cov(\Xno,\Xny)\,dy\Big)
\end{equation*}
exists and is positive, then
\begin{equation*}
\lim_{n\to\infty}\frac{1}{|U_n|}\Var\Big(\sum_{y\in\Pi\cap U_n}X_{n,y}\Big)\xrightarrow{n\to\infty}\sigma^2
\end{equation*}
for any sequence of subsets $\{U_n\}_{n\in\Nat}$ satisfying $\frac{|(\partial U_n)_{m_n}|}{|U_n|}\xrightarrow{n\to\infty}0$.
\end{cor}
\begin{rem}
The result holds for instance by taking cubic windows $U_n=C_{\mathbf{u_n}}$ with $\frac{m_n}{u_n}\xrightarrow{n\to\infty}0$.
\end{rem}
\begin{proof}
Set $p=m_n$ in the proof of Theorem~\ref{thm lim var}. Due to $m_n-$dependence, condition~\eqref{A1 thm existence var} is trivially fulfilled and therefore $\limsup\limits_{n\to\infty}\frac{|R_{n,2}|}{|U_n|}=0$. On the other hand,
\begin{align*}
\frac{|R_{n,1}|}{|U_n|}\leq\frac{|(\partial U_n)_{m_n}|}{|U_n|}\int_{\R^d}|\Cov(\Xno,\Xny)|\,dy\xrightarrow{n\to\infty}0
\end{align*}
in view of assumption~\eqref{A1 thm existence var mdep case} and the choice of $U_n$.
\end{proof}

\subsection{Application to entropy}
The results of last paragraph evince that the independence between the Poisson point process \edit{$\Pi$} and the sequence $\{X_{n,y},\,y\in\R^d_+\}_{n\in\Nat}$ is crucial to perform calculations. Therefore, we need to consider the modified estimator
\[
\widehat{\e}^*_f(B_n):=-\frac{1}{\lambda|B_n|}\sum\limits_{i\geq 1}\mathds{1}_{\{Y_i^*\in B_n\}}\log \fny(\xi_i^*),
\]
where $\Psi^*:=\{(Y_i^*,\xi^*_i)\}_{i\geq 1}$ is an independent copy of the original Poisson MPP $\Psi$. The study of the original estimator is subject of further research and it involves \edit{MPPs whose} marks depend of their location (we refer to~\cite{Paw09,HLS14,HKM14} for some investigations in this direction). Moreover, we also need to assume

\noindent
(f3) $\inf_{\eta\in\supp f}f(\eta) :=c_0>0$.

\medskip

This assumption, although being very restrictive, is usual in the context of entropy estimation (see e.g.~\cite{BDGM97}). We could substitute it by a set of slightly milder yet cumbersome assumptions and opted for the former for ease of proofs. The aim of this section is to apply Corollary~\ref{cor CLT mdep random sum} in order to obtain a CLT for $\widehat{\e}^*_f(B_n)$.
\new{
\begin{thm}\label{thm CLT EE}
Let $\{B_n\}_{n\in\Nat}$ and $\{B'_n\}_{n\in\Nat}$ be sequences of observation windows in $\R_+^d$ with $B_n=C_{\mathbf{p_n}}$, $B_n'=C_{\mathbf{m_n}}$ for some $p_n,m_n>0$. Under the conditions of Theorem~\ref{thm L2 conv EE}, there exists $a>0$ such that for any $n\in\Nat$,
\begin{equation}\label{eq bound in CLT}
\sup_{x\in\R}|F_n(x)-\Phi(x)|\leq \frac{600a\lambda(1+\lambda^2+\lambda^3)(10|B_n'|^{1/d}+11)^{2d}}{|B_n|^{1/2}},
\end{equation}
where $F_n$ is the distribution function of 
\[
\sqrt{|B_n|}\,\frac{\widehat{\e}^*_f(B_n)-\hat{\mu}_{B_n}}{\sigma_n}
\]
with
\[
\hat{\mu}_{B_n}:=-\frac{\Pi^*(B_n)}{\lambda|B_n|}\Exp\big[\log \hat{f}_{B_n'}(\xi_{\mathbf{0})}\big]
\]
and
\[
\sigma^2_n:=\lambda\Var(\log\hat{f}_{B'_n}(\xi_{\mathbf{0}}))+\lambda^2\int_{B_n'}\Cov(\log\hat{f}_{B_n'}(\xi_\mathbf{0}),\log\hat{f}_{B_n'}(\xi'_y))\,dy,
\]
where $\{\xi'_y\}_{y\in\R_+^d}$ are independent copies of $\xio$.
\end{thm}
Choosing a suitable size relation between $B_n$ and $B_n'$ leads to the desired CLT.
\begin{cor}
If the side-lengths of the observation windows satisfy $p_n=m_n^{4+\delta}$ for some $\delta>0$ and any $n\in\Nat$, then 
\[
\sqrt{|B_n|}\,\frac{\widehat{\e}^*_f(B_n)-\hat{\mu}_{B_n}}{\sigma_n}~\xrightarrow[n\to\infty]{d}~\mathcal{N}(0,1)
\]
with the uniform rate of convergence of order $m_n^{-\delta d/2}$ given in~\eqref{eq bound in CLT}.
\end{cor}
}
\vspace*{-.25in}
\new{\subsection{Proof of Theorem~\ref{thm CLT EE}}
First of all, notice that
\[
\sqrt{|B_n|}\,\frac{\widehat{\e}^*_f(B_n)-\hat{\mu}_{B_n}}{\sigma_n}=:\sum_{y\in\Pi^*\cap B_n}X_{n,y},
\]
where $X_{n,y}=\frac{1}{\sqrt{|B_n|}\sigma_n}\big(-\log\hat{f}_{B_n'+y}(\xi_y^*)+\Exp[\log\hat{f}_{B_n'}(\xio)]\big)$ is a stationary centered $m_n$-dependent random field with variance one. Our strategy will thus consist in verifying condition~\eqref{eq A prime} and computing the bound given by Corollary~\ref{cor CLT mdep random sum}. In order to do so we prove next some helpful lemmata. 
}

For the ease of reading, we use the notation $\fny$ instead of $\fny(\xi_y')$ and only refer explicitly to the argument when confusion may occur. Moreover, we assume that the conditions of Theorem~\ref{thm CLT EE} hold in the subsequent lemmata without mentioning them explicitly.

Let us begin by proving the uniform boundedness of the third moment.
\begin{lemma}\label{lemma 3.moment}
There exists a constant $c_1>0$ such that for any $y\in\R^d_+$ and $n\in\Nat$
\[
\Exp\big[\big\vert\log\fny\big\vert^3\big] \leq c_1.
\]
\end{lemma}
\begin{proof}
Due to stationarity, it suffices to show that the assertion holds for $\Exp\big[\big\vert\log\fno\big\vert^3\big]$. On the one hand, by adding and subtracting $\log\Exp[\fno]$ we have
\begin{align*}
\Exp\big[\big\vert\log\fno\big\vert^3\big]&\leq \Exp\big[\big\vert\log\fno-\log\Exp[\fno]\big\vert^3\big]+3\,\big\vert\log\Exp[\fno]\big\vert\Exp\big[\big\vert\log\fno-\log\Exp[\fno]\big\vert^2\big]\nonumber\\
&+3(\log\Exp[\fno])^2\Exp\big[\big\vert\log\fno-\log\Exp[\fno]\big\vert\big]+ \big\vert\log\Exp[\fno]\,\big\vert^3.
\end{align*}
By Corollary~\ref{thm L2 conv}, $\log\Exp[\fno]\xrightarrow{n\to\infty}\log\Exp[f(\xio)]$. \edit{In view of (f3) and s}ince $f$ is continuous, any power of this quantity is also bounded. Thus, it suffices to show that $\Exp[\vert\log\fno-\log\Exp[\fno]\vert^3]<\infty$. For $n\in\Nat$ large,~\eqref{eq mean value},~\eqref{eq min fny f} and assumption (f3) yield
\begin{equation*}
\Exp[\vert\log\fno-\log\Exp[\fno]\vert^3]\leq\frac{8\Exp\big[\big\vert\fno-\Exp[\fno]\big\vert^3\big]}{c_0^3}
\end{equation*}
for $n\in\Nat$ large, hence it suffices to prove that $\Exp[|\fno|^3]$ is finite. Due to the Campbell theorem,
\begin{align}
\Exp[\vert\fBn(\xio')\vert^3]
&=\frac{1}{\lambda^2|B_n'|^2}\Exp[F^3_n(\xio',\xi_1)]+\frac{1}{\lambda|B_n'|}\Exp[F^2_n(\xio',\xi_1)F_n(\xio',\xi_2)]\nonumber\\
&+\Exp[F_n(\xio',\xi_1)F_n(\xio',\xi_2)F_n(\xio',\xi_3)],\label{eq bound 3.moment}
\end{align}
where $\xi_1,\xi_2,\xi_3$ are independent copies of $\xio'$. Moreover, following the proof of Lemma~\ref{lemma Var} we find constants $C_\theta, K_0>0$ such that for $n\in\Nat$ large enough,
\begin{equation*}
\Exp[F^3_n(\xio',\xi_1)]\leq\frac{C_\theta^2K_0^2}{b_n^{2p}}(1+o(1))\Exp[f(\xio')],
\end{equation*}
\begin{equation*}
\Exp[F_n^2(\xio',\xi_1)F_n(\xio',\xi_2)]\leq\frac{C_\theta K_0}{b_n^{p}}(1+o(1))\Exp[f^2(\xio')],
\end{equation*}
as well as
\begin{equation*}
\Exp[F_n(\xio',\xi_1)F_n(\xio',\xi_2)F_n(\xio',\xi_3)]\leq (1+o(1))\int_Mf(\eta)^4d\upsilon_g(\eta)=(1+o(1))\Exp[f^3(\xio')].
\end{equation*}
Plugging this into~\eqref{eq bound 3.moment} we obtain
\begin{align*}
\Exp[\vert\fno\vert^3]&\leq \frac{2C_\theta^2K_0^2}{\lambda^2b_n^{2p}|B_n'|^2}\Exp[f(\xio')]+\frac{2C_\theta K_0}{b_n^{p}\lambda|B_n'|}\Exp[f^2(\xio')]+2\Exp[f^3(\xio')]
\end{align*}
for $n\in\Nat$ sufficiently large. This quantity is bounded because all expressions depending on $n$ tend to zero as $n\to\infty$.
\end{proof}

\new{
The consequent lemmata show that $\sigma_n^2$ is uniformly bounded.
}
\begin{lemma}\label{lemma cov of log}
There exists $c_{\edit{2}}>0$ such that for any $x_1,x_2\in B_n$ and $n\in\Nat$,
\[
\Cov\big(\log \fnxI,\log\fnxII\big)\leq \edit{c_2}\Cov\big(\fnxI,\fnxII\big).
\]
\end{lemma}
\begin{proof}
Adding and subtracting $\log\Exp[\fnyI]$ resp. $\log\Exp[\fnyII]$, Theorem~\ref{thm as conv KDE},~\eqref{eq min fny f} and assumption (f3) lead to
\begin{align*}
&\Cov\big(\log\fnxI,\log\fnxII\big)\\
&=\Exp[(\log\fnxI-\log\Exp[\fno])(\log\fnxII-\log\Exp[\fno])]-(\Exp[\log\fno]-\log\Exp[\fno])^2\\
&\leq\frac{4}{c_0^2}\Cov(\fnxI,\fnxII)
\end{align*}
for $n\in\Nat$ sufficiently large. The result now follows for any $n\in\Nat$ with a constant $\edit{c_2}>0$ (maybe different from $4/c_0^2$).
\end{proof}

\begin{lemma}\label{lemma cov f}
There exists $\edit{c_3}>0$ such that for any $n\in\Nat$ and $x_1,x_2\in B_n$,
\[
\Cov(\fnxI,\fnxII)\leq \frac{\edit{c_3}|( B_n'+x_1)\cap (B_n'+x_2)|}{\lambda|B_n'|^2}.
\]
\end{lemma}
\begin{proof}
Applying the Campbell theorem,
\begin{align*}
\Cov(\fnxI,\fnxII)&=\Exp[\fnxI\fnxII]-(\Exp[\fno])^2\\
&=\frac{1}{\lambda^2|B_n'|^2}\Exp\Big[\sum_{y\in\Pi\cap (B_n'+x_1)\cap (B_n'+x_2)}F_n(\xiy,\xixI')F_n(\xiy,\xixII')\Big]\\
&+\frac{1}{\lambda^2|B_n'|^2}\Exp\Big[\sideset{}{^{\neq}}\sum_{\substack{y_1\in\Pi\cap (B_n'+x_1)\\ y_2\in\Pi\cap (B_n'+x_2)}}F_n(\xiyI,\xixI')F_n(\xiyII,\xixII')\Big]-(\Exp[F_n(\xio,\xixI')])^2\\
&=\frac{|(B_n'+x_1)\cap (B_n'+x_2)|}{\lambda|B_n'|^2}\Exp[F_n(\xio,\xixI')F_n(\xio,\xixII')]\\
&+\frac{|(B_n'+x_1)\cap (B_n'+x_2)|}{\lambda|B_n'|^2}(\Exp[F_n(\xio,\xixI')])^2.
\end{align*}
Further, it follows from~\eqref{eq EF_n approx f} that for $n\in\Nat$ large enough
\begin{align*}
\Exp[F_n(\xio,\xixI')F_n(\xio,\xixII')]&=\int_{M^3}F_n(\mu,z)F_n(z,\eta)f(\mu)f(z)f(\eta)\,d\upsilon_g(\mu,z,\eta)\\
&=(1+o(1))\int_Mf(z)^3d\upsilon_g(z)=(1+o(1))\Exp[f^2(\xio)]
\end{align*}
as well as
\begin{equation*}
\Exp[F_n(\xio,\xixI')]=\int_{M^2}F_n(\mu,z)f(\mu)f(z)\,d\upsilon_g(\mu,z)=(1+o(1))\Exp[f(\xio)].
\end{equation*}
Thus the assertion holds with $\edit{c_3}=2\Exp[f^2(\xio)]+\edit{4}(\Exp[f(\xio)])^2>0$ for $n\in\Nat$ large and for any $n\in\Nat$ with maybe a different constant $\edit{c_3}>0$.
\end{proof}
\edit{Finally, by} Corollary~\ref{thm L2 conv} and analogous arguments involved in~\eqref{eq mean value}-\eqref{eq bound J_1} \edit{we have} that $\log \fBno$ converges to $\log f(\xio)$ in $L^2$. Therefore, $\Exp[\log^2\fno]\to\Exp[\log^2 f(\xio)]$ as $n\to\infty$ and since $\Exp[\log^2 f(\xio)]<L_1$ by assumption (L1), $\Exp[\log^2\fno]$ can be bounded by some constant $\tilde{L}_1>0$ uniformly on $n\in\Nat$. On the other hand, Lemma~\ref{lemma cov of log}, Lemma~\ref{lemma cov f} and the $m_n$-dependence yield
\begin{multline*}
\int_{\R^d}|\Cov(\log\fno,\log\fny)| dy=\int_{B_n'}|\Cov(\log\fno,\log\fny)| dy\nonumber\\
\leq \frac{c_1c_2}{\lambda|B_n'|^2}\int_{B_n'}|B_n'\cap(B_n'+y)|\, dy= \frac{c_1c_2}{\lambda^22^d}<\infty.
\end{multline*}
%
\new{
The next lemmata ensure that $\sigma^2_n$ can be uniformly bounded from below. }Recall that we are assuming that the density $f$ is continuous. 
\begin{lemma}\label{lemma f_n bdd as}
The estimator $\fny(\xi'_y)$ is uniformly bounded with respect to $y\in\edit{\R^d_+}$ and $n\in\Nat$ almost surely.
\end{lemma}
\begin{proof}
By stationarity it suffices to prove the assertion for $\fno(\xio)$. Note that $\xio$ is a generic mark that is independent of the MPP $\Psi$. From Theorem~\ref{thm as conv KDE} and since $M$ is compact and $f$ continuous, we have that $\fn(\eta)\xrightarrow{n\to\infty}f(\eta)\leq\norm{f}_\infty$ a.s., and hence $\fn(\eta)\leq\norm{f}_\infty+\varepsilon$ a.s. for any $\varepsilon>0$ and $n\in\Nat$. The same holds for $\fBno$.
\end{proof}

\new{
\begin{lemma}\label{lemma cov is positive}
There exists $c_4>0$ such that
\[
\liminf_{n\to\infty}\int_{B_n'}\Cov(\log\fno(\xio), \log\fny(\xi'_y))\,dy\geq c_4.
\]
\end{lemma}
}
\begin{proof}
Since $\Pi$ is a Poisson point process, we know from~\cite{BW85} that it is positively associated. On the other hand, the random variables $\{\xi'_y\}_{y\in\R^d_+}$ are positively associated as well because they are i.i.d. (see~\cite[Theorem 1.8]{BS07}). Therefore, by~\cite[Corollary 1.9]{BS07}, the random field $\{\fny(\xi_y')\}_{y\in\R^d_+}$ is positively associated. Using the characterization of positively associated random fields given in~\cite[Remark 1.4]{BS07}, this means that for any non-decreasing functions $h,g\colon\R\to\R$ such that the expectations forming the covariance $\Cov(h(\fnyI),g(\fnyII))$ exist, $\Cov(h(\fnyI),g(\fnyII))\geq 0$. In view of Lemma~\ref{lemma 3.moment} we thus have $\Cov(\log\fnyI,\log\fnyII)\geq 0$ and since $\log$ is an increasing function, the random field $\{\log\fny\}_{y\in\R^d_+}$ is also positively associated.


From Lemma~\ref{lemma f_n bdd as} we know that $\fno\leq \Vert f\Vert_\infty+\varepsilon$ a.s. for large $n\in\Nat$, and following the proof of~\cite[Theorem 5.3]{BS07} with the exponential function, we obtain
\begin{align*}
\Cov(\log\fno,\log\fny)&\geq \frac{1}{2(\Vert f\Vert_\infty+\varepsilon)^2}\Cov(\fno,\fny).
\end{align*}
Together with the calculations in the proof of Lemma~\ref{lemma cov f}, this yields 
\new{
\begin{align*}
&\int_{\R^d}\Cov(\log\fno,\log\fny)\,dy\geq\frac{\Exp[f^2(\xio)]+(\Exp[f(\xio)])^2}{4(\Vert f\Vert_\infty+\varepsilon)^2\lambda|B'_n|^2}\int_{B_n'}|B_n'\cap(B_n'+y)|\,dy\\
&=\frac{\Exp[f^2(\xio)]+(\Exp[f(\xio)])^2}{4(\Vert f\Vert_\infty+\varepsilon)^2\lambda m_n^{2d}}\Big(\int_0^{m_n}(m_n-y)dy\Big)^d=\frac{\Exp[f^2(\xio)]+(\Exp[f(\xio)])^2}{(\Vert f\Vert_\infty+\varepsilon)^2\lambda 2^{d+2}}=:c_4>0
\end{align*}
}
\edit{and the result follows with maybe a different constant $c_4$}.
\end{proof}

\new{
\begin{proof}[Proof of Theorem~\ref{thm CLT EE}]
Recall that $X_{n,y}=\frac{1}{\sqrt{|B_n|}\sigma_n}\big(-\log\hat{f}_{B_n'+y}(\xi_y^*)+\Exp[\log\hat{f}_{B_n'}(\xio)]\big)$. On the one hand, applying the Cauchy-Schwartz inequality, Lemma~\ref{lemma 3.moment} and Lemma~\ref{lemma cov is positive}, we get for $n\in\Nat$ large
\begin{equation*}
\Exp[|X_{n,\mathbf{0}}|^3]\leq\frac{8\Exp[|\log\hat{f}_{B_n'}(\xio^*)|^3]}{|B_n|^{3/2}\sigma_n^3}\leq \frac{8 c_1}{|B_n|^{3/2}\sigma_n^3}\leq \frac{8 a}{|B_n|^{3/2}}
\end{equation*}
with $a\geq c_1(\lambda c_4)^{-3/2}$.
Corollary~\ref{cor CLT mdep random sum} and the bound in Remark~\ref{rem indep PP} finally yield
\begin{align*}
\sup_{x\in\R}|F_n(x)-\Phi(x)|\leq \frac{600a\lambda(1+\lambda^2+\lambda^3)(10|B_n'|^{1/d}+11)^{2d}}{|B_n|^{1/2}}
\end{align*}
as we wanted to prove.
\end{proof}
}

\textbf{Acknowledgement.}
The authors thank the unknown referee for valuable comments that substantially improved the original manuscript.
\bibliographystyle{amsplain}
\bibliography{Rfields}

\end{document}